\documentclass[10pt]{amsart}
\usepackage{amscd,amssymb,graphicx,colortbl}

\usepackage{amsfonts}
\usepackage{amsmath}
\usepackage{amsxtra}
\usepackage{latexsym}
\usepackage[mathcal]{eucal}  
\usepackage{enumitem}

\input xy
\xyoption{all}

\usepackage[pdftex,bookmarks,colorlinks,breaklinks]{hyperref}

\newtheorem{thm}{Theorem}[section]
\newtheorem{f}[thm]{Fact} 
\newtheorem{cor}[thm]{Corollary}
\newtheorem{lem}[thm]{Lemma}

\newtheorem{prop}[thm]{Proposition}

\theoremstyle{definition}
\newtheorem{defin}[thm]{Definition}

\theoremstyle{remark}
\newtheorem{remark}[thm]{Remark}

\newtheorem{ex}[thm]{Example}

\newtheorem{question}[thm]{Question}

\numberwithin{equation}{section}

\newcommand{\delete}[1]{} 

\newcommand{\sk}{\vskip 0.3cm}

\newcommand{\ben}{\begin{enumerate}}

\newcommand{\een}{\end{enumerate}}
\newcommand{\bit}{\begin{itemize}}

\newcommand{\eit}{\end{itemize}}

\def\R {{\mathbb R}}
\def\N {{\mathbb N}}
\def\Z {{\mathbb Z}}

\def\Aut{{\mathrm Aut}\,}

\def\Homeo{{\mathrm{Homeo}}\,}

\def\B{{\mathcal{B}}}

\def\Ucal{\mathcal{U}}

\def\Iso{\operatorname{Iso}}

\def\VC{\operatorname{VC}}

\def\a{\alpha}

\def\t{\tau}
\def\s{\sigma}

\newcommand{\rank}{\operatorname{rank}} 
\newcommand{\ind}{\operatorname{ind}} 
\newcommand{\med}{\operatorname{med}}

\hypersetup{
	linkcolor=blue,
	citecolor=red,
	urlcolor=red
}

\begin{document} 

\title[]
{Tameness of actions on finite-rank median algebras} 

\author[]{Michael Megrelishvili} 
\address{Department of Mathematics,
Bar-Ilan University, 52900 Ramat-Gan, Israel}
\email{megereli@math.biu.ac.il}
\urladdr{http://www.math.biu.ac.il/$^\sim$megereli}

\thanks{Supported by the Gelbart Research Institute at the Department of Mathematics, Bar-Ilan  University}  

\subjclass[2020]{37B05, 52A01, 54H15, 20F65, 11K31}   
\keywords{Helly selection, median algebra, finite rank, median-preserving maps, tame dynamical systems, M\"obius disjointness, Sarnak conjecture}

\date{August, 2026}   

 \begin{abstract}  
We show that for every finite-rank median algebra $X$, the rank of $X$ coincides with the independence number of the family of all median-preserving maps $X \to [0,1]$. In the compact topological case, the same equality holds for the family of all continuous median-preserving maps. Combined with Rosenthal's dichotomy, this yields a generalized Helly selection principle: for every finite-rank median algebra, every uniformly bounded sequence of median-preserving real-valued maps admits a pointwise convergent subsequence whose limit is again median-preserving. 

As a dynamical application, we generalize a joint result with Eli Glasner on dendrites and prove that every continuous action of a topological group by median automorphisms on a compact finite-rank median algebra is Rosenthal representable, and hence dynamically tame. We also apply this result to the Roller--Fioravanti compactification of finite-rank topological median $G$-algebras with compact intervals, and in particular to complete finite-rank median metric spaces under continuous isometric actions. In the metrizable cascade case this gives a new source of systems satisfying Sarnak's M\"obius disjointness conjecture, including natural compactifications arising from finite-dimensional cubical geometry.   
 \end{abstract}

\maketitle

\setcounter{tocdepth}{1}
\tableofcontents

 
  \section{Introduction} \label{s:intro}
 
 Median algebras serve as a unified framework for diverse structures, from distributive lattices and median graphs to CAT(0) cube complexes and dendrites. Our aim is to establish a new link between these geometric objects and the theory of tame dynamical systems. This approach provides new cascades of a  geometric nature satisfying Sarnak's \textit{M\"obius disjointness conjecture}. 

Median algebras form a flexible framework that appears naturally in several areas, including
geometry, group theory, graph theory, combinatorics, topology, and theoretical computer science.
See, for example,  \cite{Vel-book,Roller,KM-survey,CDH,Sageev,Bow13,BowditchMedian,vanMill,GM-D,Fio19,Fioravanti20}.  
  
 Tame dynamical systems first appeared in a paper of K\"ohler \cite{Ko} (under the name \textit{regular systems}). This concept has been extensively studied and developed (see, e.g.,  \cite{Glasner06,KL,Glasner-env07,GM-rose,GM-survey14,GM-D,Glas-t,GM-tLN, 
 	GM-TC,FKY,Codenotti}). This theory serves as a bridge between low complexity topological dynamics and the low complexity Banach spaces; namely \textit{Rosenthal Banach spaces} (not containing $\ell_1$).   
 By a dynamical analogue of the  Bourgain–Fremlin–Talagrand dichotomy, a compact metrizable dynamical system is tame if and only if its enveloping semigroup is a Rosenthal compact space. 
Many remarkable naturally defined dynamical $G$-systems coming from geometry, analysis and symbolic dynamics
are tame. 

For the purposes of the present paper, Rosenthal representability should be viewed as a functional-analytic form of low dynamical complexity. It says that the system can be realized on weak-star compact subsets of duals of Banach spaces which do not contain an isomorphic copy of $\ell_1$. In the metrizable case this is equivalent to dynamical tameness. 

There is also a number-theoretic motivation. 
By a theorem of Huang--Wang--Ye \cite{HWY}, every tame compact metrizable
invertible cascade satisfies Sarnak's M\"obius disjointness conjecture. 
Consequently, every new geometric source of tame compact metrizable invertible
cascades gives a new class of systems for which the M\"obius function is
orthogonal to all continuous observations along orbits.
One of the aims of the present paper is to show that compact metrizable finite-rank median algebras, and natural metrizable compactifications attached to finite-dimensional median geometry, provide such a source. See Section \ref{s:Sarnak} for details. 
  
 One of the equivalent definitions of tame compact $G$-systems $X$ can be interpreted as the ``no independence property" (NIP) in dynamics \cite{KL,GM-tLN}. Namely, the absence of Rosenthal independent infinite sequences in the orbit $fG$ of every $f \in C(X)$. Note that a related concept of \emph{model-theoretic NIP} was introduced by Shelah \cite{Shelah04} and plays a major role in model theory. 
  
 In our earlier work \cite{GM-rose} (joint with Eli  Glasner), we established the \emph{WRN criterion} (Rosenthal Representability), which provides a functional-analytic characterization: a compact (not necessarily metrizable) $G$-system is Rosenthal representable if and only if it admits a $G$-invariant point-separating bounded family of continuous real functions with no independent infinite subsequences. 
 In \cite{GM-D}, we successfully applied this machinery to rank-$1$ structures, \textit{median pretrees}; in particular, to dendrons (note that metrizable dendrons constitute exactly the class of all dendrites). The key observation was that the canonically associated betweenness relation and  tree structure on dendrons prevents the formation of independent pairs of monotone maps.
 
 In the present paper, we extend this program to all finite ranks. We consider \emph{topological median algebras}, which form the natural generalization of many important geometric and metric structures. We establish 
 that the rank of the algebra acts as a strict bound on dynamical complexity:    
 \emph{every continuous action of a topological group $G$ on a compact finite-rank median algebra 
 $X$ by median automorphisms is Rosenthal representable (in particular, dynamically tame).} 	
 	
 This shows that the tameness observed in trees is not merely a rank-one phenomenon, but rather a
 consequence of the rigid combinatorial structure of median convexity.	
 More precisely, in rank $1$ the obstruction is ``no independent pair of monotone maps,” while if $\rank(X)=n$, then  the family of median-preserving $X \to [0,1]$ maps 
 has no independent family of size $n+1$. 
 We define \textbf{independence number} $\ind(F)$ of a family $F$ of real  functions on a set $X$, 
 which measures the maximal size of a finite  independent sequence in a function family. 
 
 Here we list some of the results of the present paper:   
 \begin{enumerate}  
 	\item (Theorem \ref{t:CrossedWallsAreIndependent}) 
 	$\ind(\mathcal{M}) = \rank(X)$ for every median algebra $X$ and the family $\mathcal{M}$ of all median-preserving maps $X \to [0,1]$. 
 	Moreover, by Lemma~\ref{l:IndEqualsDualVCSubsec}, the quantity  $\ind(\mathcal{M}) = \rank(X)$ coincides with the dual VC-dimension $\VC(\mathcal H(X)^{*})$ of the halfspace system $\mathcal H(X)$.
 
 	\item 
 	If $X$ is a finite rank \textbf{compact} median algebra, then $\ind(\mathcal{MC}) = \rank(X)$, where $\mathcal{MC}$ is the family of all continuous median-preserving maps $X \to [0,1]$ (Theorem \ref{t:RankCharacterization}).
 	  
 	\item (Theorem \ref{t:GenHelly}) Generalized Helly Selection Principle (sequential compactness of $\mathcal{M}$) for finite rank median spaces. 
  	
 	\item (Theorem \ref{t:FiniteRankTame}) Every continuous action of a topological group $G$ by median automorphisms on a finite rank compact median algebra is Rosenthal representable (in particular, dynamically tame). 
 	This can be directly applied to the Roller compactification of any finite rank median algebra (Corollary \ref{t:G-Rol}).

 \item  
 Let $X$ be a topological median algebra with compact intervals. Via a well-known canonical construction, the family of all canonical interval retractions $X \to [u,v]$ yields an important Roller-type, median-preserving compactification $\nu \colon X\to X^{RF}$ which is continuous and
 injective (see \cite{Fioravanti20,BowditchMedian}). We call it the \textit{Roller--Fioravanti compactification}. If $X$ is a topological median $G$-space with compact intervals, then by Proposition \ref{t:G-bound} the induced $G$-action on $X^{RF}$ is jointly continuous. If $X$ has finite rank, then by Corollary \ref{t:RF-tame} the RF-compactification $X^{RF}$ is a Rosenthal representable $G$-system and dynamically 
 $G$-tame. In particular this holds by Corollary \ref{t:RF} for complete finite-rank median metric
 spaces (e.g., finite-dimensional CAT(0) cube complexes with their standard median $\ell_1$-structure) and continuous isometric $G$-actions. 

 \item In the metrizable cascade case the preceding results imply M\"obius disjointness. In particular, every homeomorphic median automorphism of a compact metrizable finite-rank median algebra satisfies Sarnak's conjecture (Theorem \ref{t:SMDC}). Combining this with the intrinsic median compactification from \cite{Me-MedIntr} yields concrete classes of examples coming from Roller compactifications, MMC compactifications and finite-dimensional CAT(0) cube complexes.
 \end{enumerate}

 
\section{Preliminaries: median algebras and tame dynamical systems} 

\subsection*{Median algebras and rank}

A \textit{median algebra} is a set $X$ with a ternary operation 
$m \colon X^3 \to X$ satisfying the standard median axioms.  
Frequently we write $abc$ instead of $m(a,b,c)$. 
Recall one possible system of axioms 
(see \cite{Sholander54,Vel-book,BowditchMedian}) defining median algebras:  

\begin{itemize}
	\item [(M1)] $\s(a)\s(b)\s(c)=abc$ for every permutation $\s \in S_3$.   
	\item [(M2)] $abb=b$. 
	\item [(M3)] $(axb)xc=ax(bxc)$.    
\end{itemize}

A map $f \colon X_1 \to X_2$ between median algebras is said to be a \textit{homomorphism} or \textbf{median preserving} (MP) if
$f(xyz)=f(x)f(y)f(z)$. Equivalently: for every convex subset $C \subseteq X_2$ the preimage $f^{-1}(C)$ is convex in $X_1$. 

 For every pair $x,y \in X$ we have the \textit{interval} $[x,y]_m:=\{z \in X: xyz=z\}$. 
Usually we omit the subscript and write simply $[x,y]$, where the context is clear. Always, $[x,x]=\{x\}$, $[x,y]=[y,x]$. For every triple $x,y,z$  in $(X,m)$ we have 
$$
[x,y] \cap [y,z] \cap [x,z] = \{xyz\}.
$$ 
A subset $C \subseteq X$ is \emph{convex} if $[x,y] \subseteq C$ for all $x,y \in C$. 
Every convex subset is a subalgebra. 
The intersection of convex subsets is convex. The convex hull $co(S)$  of a subset $S \subseteq X$ is the intersection of all convex subsets of $X$ containing $S$. 

Several remarkable structures are median algebras under their natural medians. For instance distributive lattices
(e.g.\ linear orders, Boolean algebras, and power sets   $\mathcal P(S)$). 

The following is one of the key definitions in median algebras. 

\begin{defin} \label{d:MedianRank} (see e.g., \cite{Vel-book,BowditchMedian,Fioravanti20})
	The \emph{rank} of a median algebra $X$ is the supremum of the numbers $n \in \N$ such that the Boolean hypercube $\{0,1\}^n$ embeds as a median subalgebra into $X$. Notation: $\rank(X)$. 
\end{defin}

The class of finite rank algebras is closed under  taking subalgebras and finite products. The rank of the product $X_1 \times X_2$ of two median algebras is $\rank(X_1)+\rank(X_2)$. 
Surjective homomorphisms do not increase the rank. 

Rank-one algebras are 
\textit{median pretrees} (in terms of B.H. Bowditch). It is a useful treelike structure which naturally generalizes linear orders and the betweenness relation on dendrons (e.g., dendrites), simplicial and $\R$-trees. 
Important examples of algebras with rank $k \in \N$ are 
Boolean hypercubes $\{0,1\}^{k}$, Euclidean cubes $[0,1]^{k}$ and 
CAT(0) cube complexes with dimension $k$. 

Two subsets $A_1,A_2$ in a median algebra $X$ are \textit{crossing} if the following four intersections $A_1 \cap A_2, A_1 \cap A_2^c, A_1^c \cap A_2, A_1^c \cap A_2^c$ are nonempty. 

A \emph{wall} is a pair $W=\{W^0, W^1\}$ of disjoint nonempty  convex sets whose union is $X$. The sets $W^0$ and $W^1$ are called \emph{halfspaces}.
Two walls $W_1, W_2$ are \emph{crossing} if all four intersections of their halfspaces are non-empty.  
Every halfspace $H \subset X$ determines a median-preserving map
$
\chi_H \colon X\to\{0,1\},
$
and conversely every median-preserving map $f \colon X\to\{0,1\}$ is the characteristic function of a
halfspace. 

\begin{f}  \label{f:facts1} Some standard properties of median algebras. 
	\begin{enumerate} 
		\item \cite[Ch.1, 6.11]{Vel-book} A map $f \colon X_1 \to X_2$ between median algebras is median preserving (MP) if and only if it is convexity preserving (CP) in the sense of \cite[Ch.1, 1.11]{Vel-book}, meaning that for every convex subset $C \subseteq X_2$ the preimage $f^{-1}(C)$ is convex in $X$. 
		\item \cite[Theorem 2.8]{Roller} (Kakutani separation property) Any two disjoint convex sets in any median algebra are separated by a wall.
		\item \cite[Lemma 8.1.3]{BowditchMedian} Let $Q$ be a subalgebra of $X$. Then each wall of $Q$ comes from a wall of $X$. That is, any wall of $Q$ 
		is of the form $\{W^0 \cap Q, W^1 \cap Q\}$ for some wall $\{W^0, W^1\}$  of $X$.   
		\item \cite[Lemma 7.1.1]{BowditchMedian} (Helly Property) 
		Let $C_1, C_2, \cdots, C_n$ be a finite sequence of pairwise intersecting convex subsets in a median algebra. Then $\cap_{i=1}^n C_i$  is nonempty. 
		
		\item \label{t:RankCrossing} 
		\cite[Lemma 8.2.1]{BowditchMedian}, \cite[Lemma 2.5]{Fioravanti20}  
		The rank of a median algebra $X$ is equal to the supremum of $n \in \N$, where $X$ admits $n$ pairwise crossing walls. In particular, if $\rank(X)=n < \infty$ then $n$ is the maximum size of  pairwise-crossing family of walls.
	\end{enumerate}	
\end{f}

\subsection*{Topological median algebras} 

A \emph{topological median algebra} (tma) 
is a Hausdorff topological space $(X,\tau)$ equipped with a continuous median $m \colon X^3 \to X$ operation. 
If, in addition, $(X,\tau)$ is a compact space then we simply say: compact median space. 
We warn that in some publications (see, for example, \cite{vanMill,KKT}) an extra condition is assumed (namely, compact spaces with a binary convexity satisfying a separation axiom $CC_2$). 
 
Subalgebras and products (with the coordinate-wise median) of topological median algebras are themselves topological median algebras. The projection onto each coordinate is MP. 
Remarkable examples of tma are CAT(0) cube complexes  and usual cubes $[0,1]^{\kappa}$ (for every cardinal ${\kappa}$).  

Many important examples come from \textit{median metric spaces}, which play a major role in metric geometry and group theory. 
For basic information, see 
\cite{BowditchMedian, Fioravanti20,Vel-book}).   
 
\begin{f} \label{f:facts2} Some properties of topological median algebras. 
	\begin{enumerate} 		
		\item $\phi \colon X \to [x,y], \ \phi(z)=xyz$ is a continuous MP retraction for every tma $X$ and $x,y \in X$. So, if $X$ is compact then every interval $[x,y]$ is compact in $X$.  
		\item \cite[Lemma 2.7]{Fioravanti20} Let $K$ be a compact median algebra. If $C_1, \cdots, C_n$ are convex and compact in $K$ then the convex hull $co(C_1 \cup \cdots \cup C_n)$ is compact. In particular, 
		$co(F)$ is compact for every finite subset $F$ in $K$.   
		\item (\cite[12.2.4 and 12.2.5]{BowditchMedian}) Every compact finite rank median algebra is locally convex. 
		\item  
		Every compact locally convex median algebra $K$ admits a topological median embedding into a cube $[0,1]^\kappa$. The family $MC(K,[0,1])$ separates points of $K$. 
	
		
		Sketch: this follows from Chapter III in  \cite{Vel-book}; in particular from  \cite[4.13.3 and  4.16]{Vel-book}.  
		\item \cite[Lemma 12.3.4]{BowditchMedian} Let $X$ be a topological median algebra and $Y$ is its dense subalgebra. Then $\rank(Y)=\rank(X)$.
	\end{enumerate}	
\end{f}

\subsection*{Independent sequences of functions}
\label{s:ind}

Let $f_n: X \to \R$, $n \in \N$ be a uniformly bounded sequence of functions on a \emph{set} $X$. Following Rosenthal \cite{Ro} we say that this sequence is an \emph{$\ell_1$-sequence} on $X$ if there exists a constant $a >0$
such that for all $n \in \N$ and choices of real scalars $c_1, \dots, c_n$ we have
$$
a \cdot \sum_{i=1}^n |c_i| \leq ||\sum_{i=1}^n c_i f_i||_{\infty}.
$$

For every $\ell_1$-sequence $f_n$, its closed linear span in $l_{\infty}(X)$
is isomorphic to $\ell_1$ as a Banach space. 
In fact, the map $\ell_1 \to l_{\infty}(X), \ (c_n) \to \sum_{n \in \N} c_nf_n$ is a linear homeomorphic embedding.

A Banach space $V$ is said to be {\em Rosenthal} if it does not contain an isomorphic copy of $\ell_1$, or equivalently, if $V$ does not contain a sequence which is equivalent to an $\ell_1$-sequence. 
Every Asplund (in particular, every reflexive) Banach space is Rosenthal. 

A bounded sequence $f_n$ of	real valued functions on a set
$X$ is said to be \emph{independent} (see \cite{Ro}) if
there exist real numbers $a < b$ such that
$$
\bigcap_{i \in P} f_i^{-1}(-\infty,a] \cap  \bigcap_{j \in M} f_j^{-1}[b,\infty) \neq \emptyset
$$
for all finite disjoint subsets $P, M$ of $\N$. 
One may replace closed rays $(-\infty,a], [b,\infty)$ by the open rays $(-\infty,a), (b,\infty)$.  
The definition for finite sequences is analogous.  

Clearly every subsequence of an independent sequence is again independent. 
Every infinite independent sequence on a set $X$ is an $\ell_1$-sequence (see \cite{Ro}).  

Let $(X,\leq)$ be a linearly ordered set. Then any family $F$ of order preserving functions $X \to [0,1]$ is tame. Moreover, there are no independent pairs of functions in $F$, \cite{Me-Helly}.

\begin{defin} \label{d:VariousTameFamil} 
	Let $X$ be a set and $\mathcal{F} \subseteq \R^X$ a family of real-valued functions.
	\begin{enumerate}
		\item  \cite{GM-tLN} We say that $\mathcal{F}$ is a \textbf{tame} family if $\mathcal{F}$ contains no infinite  
		 uniformly bounded independent sequence.

		\item 
		Denote by $\ind(\mathcal{F})$ the supremum of integers $k$ such that $\mathcal{F}$ contains an independent finite sequence of length $k$. We call it 
		the \textbf{independence number} of $\mathcal{F}$. 
		
		\item In particular, for a family $\mathcal C\subset \mathcal P(X)$ of subsets of $X$ define
		$
		\ind(\mathcal C):=\ind(\chi_{\mathcal C}),
		$
		where
		$
		\chi_{\mathcal C}:=
		\{\chi_A \colon X\to\{0,1\}:A\in\mathcal C\}.
		$
		A subfamily $\mathcal A\subset\mathcal C$ is called
		(Rosenthal) independent if
		\[
		\bigcap_{A\in P}A\cap
		\bigcap_{B\in M}(X\setminus B)\neq\varnothing
		\]
		for every pair of finite disjoint subfamilies
		$P,M\subset\mathcal A$.
		Equivalently, every finite subfamily of $\mathcal A$ is independent.
		The independence number of $\mathcal C$ is
		\[
		\ind(\mathcal C):=
		\sup\{|\mathcal A|:\mathcal A\subset\mathcal C
		\text{ is a finite independent family}\}
		\in\mathbb N\cup\{\infty\}.
		\]
		Similarly, $\mathcal C$ is tame if $\chi_{\mathcal C}$ is tame
		in the sense of (1). 
	\end{enumerate}  
\end{defin}
 
  
  \begin{f} \label{fact:tame-family-characterization} (See, for example, \cite[Theorem 3.11]{Dulst})  
  	Let $X$ be a compact space and let $F\subset C(X)$ be a bounded family.
  	Then the following conditions are equivalent:
  	\begin{enumerate}
  		\item $F$ is tame, that is, $F$ contains no independent sequence;
  		\item $F$ contains no $\ell_1$-sequence;
  		\item every sequence in $F$ admits a pointwise convergent subsequence
  		in $\R^X$.
  	\end{enumerate}
  \end{f}

\subsection*{Tame Dynamical Systems and representations on Rosenthal spaces} \label{s:WRN}  

By a $G$-\textit{space} $X$ we mean a topological space $X$ with a continuous action $\pi \colon G \times X \to X$ of a topological group $G$.  
Let $G_d$ denote the group $G$ equipped with the discrete topology. A $G_d$-space is simply a
topological space $X$ endowed with an action of the abstract group $G$ by homeomorphisms.
Equivalently, each translation $\pi_g \colon X\to X$ is a homeomorphism, without any requirement of
joint continuity of the map $G\times X\to X$. 
 Dynamical $\Z$-systems (for $G:=\Z$) generated by a homeomorphism $T \colon K \to K$ is said to be a \textit{cascade}.  

A compact $G$-system $X$ is said to be \emph{tame} (see, for example, \cite{Glasner06,Glasner-env07,GM-rose,GM-TC}) if for every continuous real function $f \in C(X)$ its orbit $fG = \{f_g \colon  X \to \R \mid f_g(x) = f(gx), g \in G\}$ is combinatorially small; namely, if $fG$ 
is a tame family of functions on $X$. 

Recall that the \textit{enveloping semigroup} of a compact $G$-system $X$ is defined as the pointwise closure $E(G,X):=cl_p\{\pi_g: g \in G\} \subset X^X$ of all $g$-translations. In general, for compact metrizable $G$-space $X$ its enveloping semigroup might be with cardinality $2^{2^{\omega}}$. 
For instance, Bernoulli shift $\Z$-cascade $\{0,1\}^{\Z}$ is not tame. 
 In contrast, $X$ is tame iff the cardinality of $E(G,X)$ is not greater than $2^{\omega}$ (compare Fact \ref{f:RosImpliesTame}). Many naturally defined $G$-systems are tame (cf. \cite{GM-survey14,GM-TC,GM-tLN}). In Sections \ref{s:dyn} and \ref{s:Sarnak} we give some known and new concrete geometric examples of tame systems. 

Let $V$ be a Banach space and let $\operatorname{Iso}(V)$ be the topological group (with the strong operator topology) of all linear onto isometries $V \to V$. For every continuous homomorphism $h \colon G \to \operatorname{Iso}(V)$, we have a canonically induced dual continuous action on the weak-star compact unit ball $B_{V^*}$ of the dual space $V^*$. So, we get a $G$-space $B_{V^*}$.

A natural question is which continuous actions 
of $G$ on a topological space $X$ can be represented as a $G$-subspace of $B_{V^*}$ for a certain Banach space $V$ from a well-behaved class of (low-complexity) spaces.  
If $V$ is a Rosenthal Banach space then the $G$-space $X$ is said to be \textbf{Rosenthal representable}. If, in addition, $X$ is compact then the 
dynamical system $(G,X)$ is said to be WRN (\textit{Weakly Radon-Nikodym}) \cite{GM-rose}. In particular, for trivial $G$, this defines the class of WRN compact spaces, which contains the class of all \textit{Radon-Nikodym} (e.g., \textit{Eberlein}) compact spaces (recall that these are classes of all compact spaces which are representable on Asplund (resp. reflexive) Banach spaces). 

Theorem \ref{t:FiniteRankTame} below shows that every compact finite rank median space is WRN. The double arrow space is a compact linearly ordered topological space (hence, rank 1 median space) which is WRN but not RN.    
We rely on the following criterion which uses  Davis-Figiel-Johnson-Pelczy\'nski  factorization technique \cite{DFJP}.  

\begin{f} \label{WRNcriterion}   
	\cite{GM-rose,GM-survey14}  
	Let $X$ be a compact $G$-space. The following conditions are equivalent:
	\begin{enumerate}
		\item $(G,X)$ is Rosenthal representable (that is, $(G,X)$ is $\mathrm{WRN}$). 
		\item There exists a point separating bounded 
		$G$-invariant family $F \subset C(X)$ such that $F$ is a tame family.
	\end{enumerate}
\end{f}

Recall that a topological space
$K$ is a \textit{Rosenthal compactum} \cite{BFT} if it is homeomorphic to a pointwise compact 
subset of the space $\B_1(X)$ of functions of the first Baire class on a Polish space $X$. 

One of the crucial steps in the following result is 
Bourgain--Fremlin--Talagrand dichotomy \cite{BFT}. 

\begin{f} \label{f:RosImpliesTame} \cite{GM-rose,GM-AffComp13,GM-survey14,GM-tLN}   
\begin{enumerate}
	\item Every Rosenthal representable compact $G$-space is tame.
	\item For a compact \textbf{metrizable} topological $G$-space $X$ the following are equivalent: 
	\begin{enumerate}
		\item $(G,X)$ is dynamically tame. 
		\item $(G,X)$ is Rosenthal representable. 
		\item Enveloping semigroup $E(G,X) \subset X^X$ is a Rosenthal compact. 
		\item Every $p \in E(G,X)$ is a Baire 1 function $p \colon X \to X$.  
		\item The cardinality of $E(G,X)$ is not greater than $2^{\omega}$.  
	\end{enumerate} 
\end{enumerate} 	
\end{f}

Low complexity dynamical systems are representable on low complexity Banach spaces. More precisely, a compact metrizable $G$-system is tame (RN, weakly almost periodic) if and only if it is representable on a Rosenthal (Asplund, reflexive) Banach space. See \cite{GM-rose,GM-AffComp13} for more details. 

\begin{remark} \label{r:SemAct1} 
The analogues of Facts~\ref{WRNcriterion} and~\ref{f:RosImpliesTame} for
semigroup actions hold under the standard natural adaptations.	
 The corresponding Rosenthal-representation statements
can be formulated for semigroup actions with the appropriate operator representation framework replacing the group $\Iso(V)$ by the monoid $\Theta(V)$ of all contractive linear operators of $V$ into itself. See, for example, \cite[Remark 1.6]{GM-rose} and especially
\cite{GM-AffComp13} with general semigroup setting. 
The tameness of actions on dendrites is true also for semigroup actions but under the extra-condition that every translation is \textit{monotone} (this is automatic for group actions on dendrites), \cite[Remark 3.18]{GM-D}. This remark will be used below in Remark \ref{r:SemigroupActions2}. 	
\end{remark}

\section{Independence number and tameness in the family of MP functions}

\begin{defin}
	For a median algebra $X$ denote by $\ind(X)$ the independence number of the set $\mathcal{H}(X)$ of all halfspaces in  $X$. 	
\end{defin}

\begin{thm}[Characterization of rank via independence number]
	\label{t:CrossedWallsAreIndependent}
	Let $X$ be a median algebra. Then:
	\begin{enumerate}
		\item A family $\mathcal F\subset\mathcal H(X)$ of halfspaces is
		pairwise crossing if and only if it is independent. In particular,
		$$\rank(X)=\ind(X)=\ind(\mathcal H(X)).$$
		
		\item 
		$\rank(X)=\ind(\mathcal M)$, where 
		$\mathcal M=\mathcal M(X,[0,1])$ be the family of all
		median-preserving maps $X\to[0,1]$. 
	\end{enumerate}
\end{thm}

\begin{proof}
	(1) Suppose first that $\mathcal F\subset\mathcal H(X)$ is pairwise
	crossing. Let $A_1,\ldots,A_k\in\mathcal F$, and let
	$P,M\subset\{1,\ldots,k\}$ be disjoint. Consider the finite family
	of convex sets
	\[
	\mathcal C:=\{A_i:i\in P\}\cup\{A_j^c:j\in M\}.
	\]
	Since the halfspaces $A_1,\ldots,A_k$ are pairwise crossing, the
	members of $\mathcal C$ are pairwise intersecting. By the Helly
	property (Fact~2.2(4)),
	\[
	\bigcap_{i\in P}A_i\cap\bigcap_{j\in M}A_j^c\ne\varnothing.
	\]
	Thus every finite subfamily of $\mathcal F$ is independent, and hence
	$\mathcal F$ is independent.
	
	Conversely, suppose that $\mathcal F$ is independent. Let
	$A,B\in\mathcal F$ be distinct. Applying independence to the four
	possible membership patterns for $\{A,B\}$, we obtain
	\[
	A\cap B\ne\varnothing,\qquad
	A\cap B^c\ne\varnothing,\qquad
	A^c\cap B\ne\varnothing,\qquad
	A^c\cap B^c\ne\varnothing.
	\]
	Hence $A$ and $B$ are crossing. Thus $\mathcal F$ is pairwise crossing.

	(2) Since $\chi_H\in\mathcal M$ for every $H\in\mathcal H(X)$,
	part~(1) gives
	\[
	\rank(X)=\ind(\mathcal H(X))\leq \ind(\mathcal M).
	\]
	It remains to prove the reverse inequality.
		
	For the converse, let $f_1,\ldots,f_k\in\mathcal M$ be an independent
	family. Choose witnessing constants $a<b$. Since the $f_i$ are
	$[0,1]$-valued, necessarily
	$
	0\le a<b\le 1.
	$
	Fix $t\in(a,b)$ and put
	\[
	H_i:=f_i^{-1}([0,t]),\qquad 1\le i\le k.
	\]
	Since $[0,t]$ and $(t,1]$ are convex subsets of the median algebra
	$[0,1]$, and $f_i$ is median-preserving, both
	\[
	H_i=f_i^{-1}([0,t])
	\quad\text{and}\quad
	H_i^c=f_i^{-1}((t,1])
	\]
	are convex. Moreover, independence of the $f_i$ implies that both
	sets are nonempty. Hence $H_i$ is a halfspace of $X$.
	
	We claim that $H_1,\ldots,H_k$ are independent. Indeed, let
	$P,M\subset\{1,\ldots,k\}$ be disjoint. By independence of the
	$f_i$ there exists $x\in X$ such that
	\[
	f_i(x)\le a\quad (i\in P),
	\qquad
	f_j(x)\ge b\quad (j\in M).
	\]
	Since $a<t<b$, it follows that
	\[
	x\in
	\bigcap_{i\in P}H_i\cap
	\bigcap_{j\in M}H_j^c.
	\]
	Thus $\{H_1,\ldots,H_k\}$ is an independent family of halfspaces.
	Consequently,
	$
	k\le \ind(\mathcal H(X)).
	$
	Since this holds for every finite independent family in $\mathcal M$,
	\[
	\ind(\mathcal M)\le \ind(\mathcal H(X))=\rank(X).
	\]
\end{proof}

See also a characterization in terms VC-dimension (Lemma \ref{l:IndEqualsDualVCSubsec} and Remark \ref{r:DualVCmedian}). 

Below we denote by $\mathcal{MC} = \mathcal{MC}(X, [0,1])$ the class of all continuous median-preserving maps $f \colon X \to [0,1]$ on a tma $X$. 

\begin{thm} \label{t:RankCharacterization}  \ 
	\begin{enumerate}
		\item $\ind(\mathcal{MC}) \leq \rank(X)$ for every topological median algebra $X$.
		\item $\ind(\mathcal{MC}) = \rank(X)$ for every  finite rank \textbf{compact} topological median algebra $X$.
	\end{enumerate}  	    
\end{thm}    
\begin{proof} 
	(1) By Theorem \ref{t:CrossedWallsAreIndependent} $\ind(\mathcal{M}) = \rank(X)$, where $\mathcal{M} = \mathcal{M}(X, [0,1])$. Since $\mathcal{MC} \subseteq \mathcal{M}$, we have (for every topological median algebra) 
	$$
	\ind(\mathcal{MC}) \leq \ind(\mathcal{M}) = \rank(X).
	$$ 
	
	(2) It is enough to show that $\rank(X) \leq \ind(\mathcal{MC}(X,[0,1]))$ for compact finite rank $X$. 

Let $\rank(X)=k \in \N$. 
By Definition~\ref{d:MedianRank}, there exists an embedding of median algebras
$
\iota \colon \{0,1\}^k \hookrightarrow X .
$
Let
$
Q:=\iota(\{0,1\}^k)\subset X.
$
Then $Q$ is a \emph{finite} median subalgebra of $X$. For each $1\le j\le k$, let
\[
A_j:=\iota(\{x\in\{0,1\}^k:\ x_j=0\}),\qquad
B_j:=\iota(\{x\in\{0,1\}^k:\ x_j=1\}),
\]
and set $W_j:=\{A_j,B_j\}$.
Then $W_1,\dots,W_k$ are walls in $Q$ and they are pairwise crossing (this is immediate for the coordinate walls of $\{0,1\}^k$).

By the separation property of walls in median subalgebras (Fact \ref{f:facts1}.3), each wall $W_j$ in $Q$ extends to a wall 
$\widetilde W_j=\{H^0_j,H^1_j\}$ in $X$
such that
\[
H^0_j\cap Q=A_j\qquad\text{and}\qquad H^1_j\cap Q=B_j.
\]
In particular, $A_j\subseteq H^0_j$ and $B_j\subseteq H^1_j$. Since $H^0_j$ and $H^1_j$ are convex in $X$,
it follows that
\[
\operatorname{co}(A_j)\subseteq H^0_j
\qquad\text{and}\qquad
\operatorname{co}(B_j)\subseteq H^1_j,
\]
and therefore $\operatorname{co}(A_j)\cap \operatorname{co}(B_j)=\varnothing$ for every $j$.
Moreover, as $X$ is compact and $A_j,B_j$ are finite, Fact~\ref{f:facts2}.2 yields that $\operatorname{co}(A_j)$ and
$\operatorname{co}(B_j)$ are compact convex subsets of $X$.
		
	Every compact finite-rank algebra is locally convex (Fact \ref{f:facts2}.3) and has compact intervals (Fact \ref{f:facts2}.1).  
	Hence we can apply the functional separation property $FS_4$
	(see \cite[Proposition III.4.13.3]{Vel-book}) to the disjoint compact convex sets
	$\operatorname{co}(A_j)$ and $\operatorname{co}(B_j)$. 	
	This yields a continuous separating map $f_j\colon X\to[0,1]$ such that 
	$f_j(\operatorname{co}(A_j))\subseteq [0,\tfrac13]$ and
	$f_j(\operatorname{co}(B_j))\subseteq [\tfrac23,1]$,
	which is convexity-preserving in the sense of \cite{Vel-book}.
	Consequently, $f_j$ is median-preserving by Fact~\ref{f:facts1}.1.

We now verify that the MP functions $f_1,\dots,f_k$ from $\mathcal{MC}$ form an independent family.
Fix $a=\tfrac13$ and $b=\tfrac23$. 
Let $P,M\subseteq\{1,\dots,k\}$ be arbitrary disjoint sets. Choose any $\sigma\in\{0,1\}^k$ such that
$\sigma_i=0$ for $i\in P$ and $\sigma_i=1$ for $i\in M$, and let $x_\sigma:=\iota(\sigma)\in Q\subseteq X$.
Then, by construction,
\[
x_\sigma\in \bigcap_{i\in P}A_i\ \cap\ \bigcap_{j\in M}B_j.
\]
Since $A_i\subseteq \operatorname{co}(A_i)$ and $B_j\subseteq \operatorname{co}(B_j)$, and since
$f_i(\operatorname{co}(A_i))\subseteq [0,a]$ and $f_j(\operatorname{co}(B_j))\subseteq [b,1]$, we obtain
\[
x_\sigma\in \bigcap_{i\in P} f_i^{-1}((-\infty,a])\ \cap\ \bigcap_{j\in M} f_j^{-1}([b,\infty)).
\]
As $P$ and $M$ were arbitrary, this proves independence. Hence
\[
\ind(\mathcal{MC}(X,[0,1]))\ge k=\rank(X).
\]
This completes the proof. 
\end{proof}  
   
Recall that a median algebra has \textit{subinfinite rank} in the sense of Bowditch  \cite[page 116]{BowditchMedian} 
if any set of pairwise-crossing halfspaces is finite 
($\omega$-\textit{dimension} in terms of Roller \cite{Roller} and Guralnik \cite[Definition 2.2]{Gur04}). It is a natural generalization of finite rank spaces.  

\begin{remark}  \label{r:B-tame} 	
By Theorem \ref{t:CrossedWallsAreIndependent}.1, a family of halfspaces in a median algebra is independent if and only if it is
pairwise crossing. Therefore, $X$ has subinfinite rank if and only if the family of halfspaces
$H(X)$ is tame in the Boolean sense; equivalently, the family $
\{\chi_H \colon X\to\{0,1\}: H\in H(X)\}
$ of characteristic functions contains no infinite independent sequence.  	 
\end{remark}

\begin{ex}[Not finite rank but subinfinite] \label{p:SubinfiniteNotFiniteRank}
	There exists a median algebra $X$ which is Boolean-tame (equivalently, has subinfinite rank) but has $\rank(X)=\infty$.
\end{ex}

\begin{proof}
	For each $n\in\N$ let
	\(
	Q_n:=\{0,1\}^n
	\) 
	be the Boolean $n$-cube as a median algebra, and let $o_n:=(0,\dots,0)\in Q_n$.
	Define the \emph{wedge} (bouquet)
	\[
	X \ :=\ \bigvee_{n\ge 1} Q_n= \bigvee_{n\ge 1}\{0,1\}^n
	\]
	to be the disjoint union $\bigsqcup_{n\ge 1} Q_n$ where all basepoints $o_n$ are identified to a single point $o\in X$. 

	We equip $X$ with a median operation as follows. If $x,y,z$ all belong to the same cube $Q_n$ (viewed inside $X$),
	set $m(x,y,z)$ to be their median in $Q_n$. Otherwise, let $Q_n$ be a cube containing at least two of the points
	(say $x,y\in Q_n$; necessarily $o\in Q_n$), and define
	$
	m(x,y,z)\ :=\ m_{Q_n}(x,y,o).
	$
	If $x,y,z$ lie in three distinct cubes, put $m(x,y,z):=o$. 
	Equivalently, $X$ is the wedge of the pointed median graphs $(Q_n,o_n)$, and the above ternary
	operation is the usual median operation on this wedge; in particular it satisfies the median axioms.
	
	\smallskip
	
	\noindent\emph{Step 1: $\rank(X)=\infty$.}
	For every $n$, the natural inclusion $Q_n\hookrightarrow X$ is a median embedding, hence $\rank(X)\ge n$ for all $n$.
	Therefore $\rank(X)=\infty$.
	
	\smallskip
	
	\noindent\emph{Step 2: $X$ has subinfinite rank.}  Let $H \subseteq X$ be a halfspace. Then exactly one of
	$H$ and $H^c$ contains the wedge point $o$. 
	Assume first that $o \notin H$ (so $o \in H^c$). We claim that then $H$ is contained in a
	single cube $Q_n$. Indeed, if $H$ contained points $x \in Q_n \setminus \{o\}$ and
	$y \in Q_m \setminus \{o\}$ with $m \ne n$, then convexity of $H$ would force
	$[x,y] \subseteq H$, but the interval $[x,y]$ in the wedge passes through $o$, contradicting
	$o \notin H$.
	
	Therefore, for every halfspace $H \subseteq X$, one of the two halfspaces $H$ or $H^c$
	is contained in a single cube $Q_n$. Equivalently, every wall of $X$ is supported on a
	single cube.
	
	Now let $\{H_1,H_1^c\}$ and $\{H_2,H_2^c\}$ be two walls supported on distinct cubes
	$Q_n$ and $Q_m$, where $n \ne m$. Without loss of generality, assume
	$H_1 \subseteq Q_n$ and $H_2 \subseteq Q_m$. Then $H_1 \cap H_2 = \varnothing$, so these
	two walls cannot be crossing. Hence any pairwise crossing family of walls in $X$ must be
	supported on a single cube $Q_n$. But $Q_n$ has rank $n$, so it admits no pairwise
	crossing family of walls of size $>n$. Therefore every pairwise crossing family of walls in
	$X$ is finite. 
\end{proof}

\subsection*{Independence number and dual VC-dimension}\label{s:DualVCDim} 

	Theorems~\ref{t:CrossedWallsAreIndependent} and~\ref{t:RankCharacterization} show that rank of median algebras 
is closely related to independence/shattering complexity of the family of walls. 
Compare with the role of VC dimension and NIP in \cite{SimonNIP} and with related ``no-independence''
conditions in dynamics \cite{KL,GM-tLN}. Note that the model-theoretic \emph{NIP} (``no independence property'')
was introduced by Shelah \cite{Shelah04}. Tame dynamical systems is the dynamical analog of NIP. 

We recall two definitions: VC-dimension (Vapnik–Chervonenkis) and the dual set system. 
For background and further references we refer, for example, to \cite[\S 6.1]{SimonNIP} and \cite[Definition~2.9]{DCG-dualVC}.  
For the reader's convenience we include a short 
discussion, and explain how these
notions relate to the independence number. 
Variants of this viewpoint appear in the literature in connection with learning theory and NIP. 

\begin{defin}[Vapnik–Chervonenkis dimension] \label{d:VCdimSubsec}
	Let $X$ be a set and let $\mathcal C\subseteq\mathcal P(X)$ be a family of subsets of $X$.
	A finite set $S\subseteq X$ is \emph{shattered} by $\mathcal C$ if 
	$
	\{C\cap S:\ C\in\mathcal C\}=\mathcal P(S).
	$
	The \emph{VC-dimension} of $\mathcal C$ is
	\[
	\VC(\mathcal C):=\sup\{|S|:\ S\subseteq X\ \text{finite and shattered by }\mathcal C\}
	\in\mathbb N\cup\{\infty\}.
	\]
\end{defin}

\begin{defin}[Dual set system]\label{d:DualSetSystemSubsec}
	Let $X$ be a set and let $\mathcal C\subseteq\mathcal P(X)$. Define a map
	\[
	\Phi:\ X\longrightarrow \mathcal P(\mathcal C),
	\qquad
	\Phi(x):=\{\,C\in\mathcal C:\ x\in C\,\}.
	\]
	For $x\in X$ put $R_x:=\Phi(x)\subseteq\mathcal C$. The \emph{dual set system} of $\mathcal C$
	is the family
	\[
	\mathcal C^{*}:=\{R_x:\ x\in X\}\subseteq\mathcal P(\mathcal C),
	\]
	viewed as a set system on the ground set $\mathcal C$.
	See, for example, \cite[\S 6.1]{SimonNIP} or \cite[Definition 2.9]{DCG-dualVC}.
\end{defin}

\begin{lem}[Independence number equals the dual VC-dimension] \label{l:IndEqualsDualVCSubsec}
	For every set system $\mathcal C\subseteq\mathcal P(X)$ one has
	$
	\ind(\mathcal C)\ =\ \VC(\mathcal C^{*}).
	$
\end{lem}

\begin{proof}
	Let $F=\{C_1,\dots,C_k\}\subseteq\mathcal C$. 
	By Definition~2.4(3), the family $F$ is independent if and only if for every decomposition
	$
	F=T\sqcup(F\setminus T)
	$
	there exists $x\in X$ such that
	\[
	x\in C_i \quad \text{for all } C_i\in T,
	\qquad\text{and}\qquad
	x\notin C_j \quad \text{for all } C_j\in F\setminus T.
	\]
	Equivalently, for every subset $T\subseteq F$ there exists $x\in X$ such that
	\(
	C_i\in T \iff x\in C_i .
	\)
	
	On the other hand, $F$ is shattered by the dual system $\mathcal C^{*}$ if and only if for every
	$T\subseteq F$ there exists $x\in X$ with
	\[
	R_x\cap F=T,
	\qquad\text{where }R_x=\{C\in\mathcal C:\ x\in C\}.
	\]
	Since $R_x\cap F=\{C_i\in F:\ x\in C_i\}$, these two conditions are equivalent.
	Therefore the maximal size of an independent family in $\mathcal C$ equals the maximal size of a
	subset of $\mathcal C$ shattered by $\mathcal C^{*}$, i.e.\ $\ind(\mathcal C)=\VC(\mathcal C^{*})$.
\end{proof}

The number $\VC(\mathcal C^{*})$ is often called the
\emph{dual VC-dimension} of $\mathcal C$.  
See \cite[\S 6.1]{SimonNIP}.

\begin{remark}[Rank via VC-dimension] \label{r:DualVCmedian}
	Let $X$ be a median algebra and let $\mathcal H(X)\subseteq\mathcal P(X)$ be the family of all halfspaces. 
%
	If $X$ has finite rank $n$, then by Theorem~\ref{t:CrossedWallsAreIndependent},  
	$
	\ind(\mathcal H(X))=\rank(X)=n.
	$
	Therefore, by Lemma~\ref{l:IndEqualsDualVCSubsec},
	$
	\VC(\mathcal H(X)^{*})=\rank(X)=n,
	$
	that is, the rank of $X$ coincides with the dual VC-dimension of the halfspace system $\mathcal H(X)$.
\end{remark}

\section{Generalized Helly selection principle}
\label{s:Helly}

The following result is purely algebraic and requires no topological
assumptions on $X$. For the specific case of linearly ordered sets, which
naturally form rank-one median algebras under the standard betweenness
relation, this theorem was established in \cite{Me-Helly}, generalizing
the classical Helly selection theorem for functions on $X\subset\R$.
Although the main emphasis of this paper is on finite-rank median algebras,
the result holds more generally for median algebras of subinfinite rank.

\begin{thm}[Helly Selection Principle for median spaces and MP functions]
	\label{t:GenHelly}
	Let $X$ be a median algebra of subinfinite rank (in particular, of finite
	rank). Let $(f_n)_{n\in\N}$ be a uniformly bounded sequence of
	median-preserving maps $f_n \colon X\to\R$. Then $(f_n)$ admits a pointwise
	convergent subsequence, and its pointwise limit $f \colon X\to\R$ is again
	median-preserving.
	
	Equivalently, for every compact interval $I\subset\R$, the family
	$\mathcal M(X,I)$ is sequentially compact in the pointwise topology of
	$I^X$. In particular, $\mathcal M(X,[0,1])$ is sequentially compact.
\end{thm}

\begin{proof}
	Suppose, towards a contradiction, that $(f_n)$ has no pointwise convergent
	subsequence. We use the corresponding part of the proof of Rosenthal's
	dichotomy theorem \cite[Theorem~1]{Ro}. For every infinite
	$L\subseteq\N$ put
	\[
	\delta(L):=\sup_{x\in X}
	\left(\limsup_{n\in L}f_n(x)-\liminf_{n\in L}f_n(x)\right).
	\]
	By \cite[Lemma~5]{Ro}, there exists an infinite $Q\subseteq\N$ such that
	$\delta(L)=\delta(Q)>0$ for every infinite $L\subseteq Q$. Put
	$\varepsilon:=\delta(Q)/2$. By \cite[Lemma~6]{Ro}, there exist an infinite
	$M'\subseteq Q$ and a rational number $r$ such that for every infinite
	$L\subseteq M'$ there is $x\in X$ satisfying
	$\limsup_{n\in L}f_n(x)>r+\varepsilon$ and
	$\liminf_{n\in L}f_n(x)<r$.
	
	For $n\in M'$ set
	$A_n:=\{x\in X:f_n(x)>r+\varepsilon\}$ and
	$B_n:=\{x\in X:f_n(x)<r\}$. The preceding property implies that no
	subsequence of the sequence of pairs $(A_n,B_n)_{n\in M'}$ is convergent
	in Rosenthal's sense. Hence, by \cite[Theorem~2]{Ro}, there is an infinite
	subsequence $(f_{n_k})$ which is independent. Let $a<b$ be corresponding
	witnessing constants.
	
	Choose $t\in(a,b)$ and put
	$H_k:=f_{n_k}^{-1}((-\infty,t])$. Since both $(-\infty,t]$ and
	$(t,\infty)$ are convex in the median algebra $\R$, the sets $H_k$ and
	$H_k^c=f_{n_k}^{-1}((t,\infty))$ are convex. Independence implies that
	both are nonempty, so $H_k$ is a halfspace of $X$.
	
	Moreover, $(H_k)$ is an independent family of halfspaces. Indeed, for
	finite disjoint $P,M\subset\N$, independence of $(f_{n_k})$ gives
	$x\in X$ such that $f_{n_i}(x)\le a$ for $i\in P$ and
	$f_{n_j}(x)\ge b$ for $j\in M$. Since $a<t<b$, we have
	$x\in\bigcap_{i\in P}H_i\cap\bigcap_{j\in M}H_j^c$.
	By Theorem~\ref{t:CrossedWallsAreIndependent}(1), the halfspaces $H_k$
	are pairwise crossing. This contradicts the assumption that $X$ has
	subinfinite rank.
	
	Therefore $(f_n)$ has a pointwise convergent subsequence
	$f_{n_k}\to f$. Finally, for every $x,y,z\in X$ we have
	$f_{n_k}(m(x,y,z))=\med(f_{n_k}(x),f_{n_k}(y),f_{n_k}(z))$.
	Passing to the limit and using continuity of the median on $\R$, we obtain
	$f(m(x,y,z))=\med(f(x),f(y),f(z))$. Hence $f$ is median-preserving.
\end{proof}

Theorem~\ref{t:GenHelly} shows that Helly-type subsequence selection for
median-preserving maps is governed by the combinatorial tameness of median
convexity, rather than by the presence of a linear order. In particular,
the classical one-dimensional Helly phenomenon extends to all
subinfinite-rank median algebras.

\begin{remark}
	There is also a shorter route to the extraction step in the proof above.
	The following observation applies to functions on an arbitrary set:
	every bounded $\ell_1$-sequence contains an independent subsequence.
	
	Indeed, let $(f_n)$ be a bounded $\ell_1$-sequence on a set $X$, regard
	$X$ as discrete, and let $\widetilde f_n\in C(\beta X)$ be the continuous
	extensions to the Stone--Čech compactification. Since $X$ is dense in
	$\beta X$, for every finite choice of scalars $c_1,\ldots,c_m$ one has
	\[
	\left\|\sum_{i=1}^m c_i\widetilde f_i\right\|_{C(\beta X)}
	=
	\left\|\sum_{i=1}^m c_i f_i\right\|_{\ell_\infty(X)}.
	\]
	Hence $(\widetilde f_n)$ is again an $\ell_1$-sequence. By
	Fact~\ref{fact:tame-family-characterization}, the bounded family
	$\{\widetilde f_n:n\in\N\}\subset C(\beta X)$ is not tame and therefore
	contains an independent subsequence.
	
	This independence already holds on $X$. Indeed, using the equivalent
	open-ray formulation of independence, every prescribed finite pattern
	defines a nonempty open subset of $\beta X$; by density it meets $X$.
	Thus the corresponding subsequence of $(f_n)$ is independent on $X$.
	
	Consequently, one may alternatively prove Theorem~\ref{t:GenHelly} by
	first applying Rosenthal's dichotomy \cite[Theorem~1]{Ro} to obtain an
	$\ell_1$-subsequence, then using the preceding observation to extract an
	independent subsequence, and finally applying the halfspace argument from
	the proof above. The same extraction observation is also presented in the
	recent preprint of Glasner \cite[Lemma~9]{Glasner-new}.
\end{remark}

\begin{cor}\label{c:HellyK}
	Let $X$ be a median algebra of subinfinite rank and let $K$ be a
	compact metrizable locally convex median algebra. Then the family
	$\mathcal M(X,K)$ of all median-preserving maps $X\to K$ is
	sequentially compact in the pointwise topology of $K^X$.
	
	In particular, this holds when $K$ is a compact metrizable
	finite-rank median algebra.
\end{cor}

\begin{proof}
	By Fact~\ref{f:facts2}(4), the family $\mathcal{MC}(K,[0,1])$
	separates points of $K$. Since $K$ is compact metrizable, there exists
	a countable point-separating family
	$\{\varphi_j:j\in\N\}\subset\mathcal{MC}(K,[0,1])$.
	
	Let $(F_n)$ be a sequence in $\mathcal M(X,K)$. For every $j$ the
	sequence $(\varphi_j\circ F_n)$ consists of median-preserving maps
	$X\to[0,1]$. By Theorem~\ref{t:GenHelly}, and a standard diagonal
	argument, there exists a subsequence $(F_{n_k})$ such that
	$\varphi_j\circ F_{n_k}$ converges pointwise on $X$ for every $j$.
	
	The diagonal map
	$\Phi \colon K\to[0,1]^{\N}$, $\Phi(y):=(\varphi_j(y))_{j\in\N}$,
	is a continuous injection. Since $K$ is compact, $\Phi$ is a
	topological embedding. Hence, for every $x\in X$, the sequence
	$(F_{n_k}(x))$ converges in $K$. Denote its limit by $F(x)$.
	
	Finally, continuity of the median operation on $K$ and
	$F_{n_k}(m(x,y,z))=m(F_{n_k}(x),F_{n_k}(y),F_{n_k}(z))$ imply, after
	passing to the limit, that $F$ is median-preserving.
\end{proof}

Theorem~\ref{t:GenHelly} and Corollary~\ref{c:HellyK} show that
Helly-type subsequence selection for median-preserving maps is governed
by the combinatorial tameness of median convexity, namely by the absence
of an infinite independent family of halfspaces, rather than by the
presence of a linear order. Thus the classical one-dimensional Helly
principle extends to the substantially broader class of subinfinite-rank
median algebras.

\section{Dynamical tameness of group actions} 
\label{s:dyn}

Denote by $\Aut(X)$ the group of all median automorphisms of  a median algebra $X$. 
By a (topological) \textit{median $G$-algebra} $X$ we mean a (topological) median algebra $X$ equipped with a median preserving (topological) group (continuous) action $\pi \colon G \times X \to X$. In this case we have a natural homomorphism $h_{\pi} \colon G \to \Aut(X)$. If $X$ is a compact median $G$-algebra then $h_{\pi}$ is continuous where $\Aut(X)$ is equipped with the compact-open topology. 

\begin{defin} 
	Let $X$ be a topological median $G$-algebra.  
	We say that $X$ is:
	\begin{enumerate}
		\item \textit{Rosenthal representable} if the $G$-space $X$ is Rosenthal representable.  
		\item \textit{Dynamically tame}, if $X$ is compact and the $G$-system $X$ is tame. 
	\end{enumerate} 
\end{defin}
  
A sufficient condition for dynamical tameness in the case of a compact locally convex median algebra $X$ is that the family $\mathcal{MC}(X, [0,1])$ of all continuous MP maps is a tame family (e.g. has finite independence number). 
In fact, it is enough that the orbit $fG$ is a tame family for every $f \in \mathcal{MC}(X, [0,1])$.

\begin{thm}[Finite rank implies dynamical tameness] \label{t:FiniteRankTame}
	Let $X$ be a compact median algebra of finite rank $n$. Then for every continuous median preserving action of a topological group $G$ on $X$ the dynamical system $(G, X)$ is Rosenthal representable (in particular, dynamically tame). Thus, there exists a Rosenthal Banach space $V$, a continuous homomorphism $h \colon G \to \Iso(V)$ and a weak-star $G$-embedding $\a \colon X \to V^*$.  
\end{thm}
\begin{proof}
	We apply the WRN Criterion (Fact \ref{WRNcriterion}) to the family $\mathcal{F} = \mathcal{MC}(X, [0,1])$.
	
	\noindent \textit{$G$-Invariance:} The composition of a median morphism and an automorphism is a median morphism. Thus $\mathcal{F}$ is invariant.
	 
	\noindent \textit{Point Separation:} By Fact \ref{f:facts2}.3 every compact finite rank median algebra is locally convex. Hence, $X$ is locally convex.  
	  By Fact \ref{f:facts2}.4 
	 every compact locally convex median algebra $X$ embeds (topologically and algebraically) into a Tychonoff cube $[0,1]^\kappa$ (compact median algebra). The coordinate projections of this embedding are continuous median-preserving maps. Thus, the family $\mathcal{F}$ separates points of $X$.
	   
	\noindent \textit{Tameness of $\mathcal{F}$:} By Theorem  \ref{t:RankCharacterization}, the size of any independent sequence in $\mathcal{F}$ is bounded by $\rank(X) = n$. Since $n$ is finite, $\mathcal{F}$ contains no infinite independent sequence.
	
	Therefore, $(G, X)$ is WRN (Rosenthal representable)  by Fact \ref{WRNcriterion}. In particular, the $G$-system is dynamically tame according to Fact \ref{f:RosImpliesTame}.  
\end{proof}

\begin{remark}[Semigroup actions]
	\label{r:SemigroupActions2}
	The finite-rank tameness result itself admits a natural semigroup version.
	Indeed, let \(S\) be a topological semigroup acting continuously on a compact
	finite-rank median algebra \(K\) by median endomorphisms, that is,
	\[
	s\,m(x,y,z)=m(sx,sy,sz)
	\qquad (s\in S,\ x,y,z\in K).
	\]
	Then for every \(f\in MC(K,[0,1])\) and every \(s\in S\), the translated
	function \((fs)(x):=f(sx)\) again belongs to \(MC(K,[0,1])\). Therefore the same argument as in Theorem~\ref{t:FiniteRankTame}
	(cf. Remark~\ref{r:SemAct1}) shows that the compact \(S\)-system \(K\) is dynamically tame. Thus the tameness conclusion of Theorem~\ref{t:FiniteRankTame}
	has a natural semigroup version.  
\end{remark}

As a consequence, for compact finite rank median $G$-spaces the conditions of Fact \ref{f:RosImpliesTame} are satisfied. 
 An additional significant consequence of Theorem~\ref{t:FiniteRankTame} is the structural rigidity it imposes on minimal subsystems. 
  As established by Glasner \cite{Glasner06}, in a tame compact $G$-system, every distal minimal $G$-subsystem is necessarily equicontinuous. By Theorem \ref{t:FiniteRankTame}, this happens in finite rank compact median $G$-spaces $X$.

 \subsection*{Roller compactifications for median $G$-spaces}
 
 In any median algebra $X$ the set $\mathcal{H}(X)$ of all halfspaces separates the points 
 by Fact \ref{f:facts1}.2. Therefore the diagonal map 
 $$
 \iota \colon X \to \{0,1\}^{\mathcal{H}(X)}
 $$
 is an injective MP map. Passing to the closure $X^R:=cl(\iota(X))$ 
 we get the \textit{Roller compactification} $\iota \colon X \to X^R$, equivalently describable as the subspace consisting of ultrafilters on ${\mathcal{H}(X)}$, or using a double dual construction. 
 It has many applications.  
 See \cite{Roller,Fioravanti20,BowditchMedian} for details and alternative definitions.

 \begin{cor} \label{t:G-Rol}
 	Let $\iota \colon X \to X^R \subset \{0,1\}^{\mathcal{H}(X)}$ be the Roller compactification of a  
 	finite rank median algebra $X$. Assume that an abstract discrete group $G$ acts on $X$ by median transformations. Then the induced compact dynamical system $(G,X^R)$ is Rosenthal representable.  
 \end{cor}
 \begin{proof} Since the action is median preserving, 
  for every halfspace $H \in \mathcal{H}(X)$ and every $g \in G$ we have $gH \in \mathcal{H}(X)$. This implies that $G$ acts on the compact median space $X^R$ by continuous median automorphisms such that $\iota$ is a $G$-map. 
By Fact \ref{f:facts2}.5 we have the coincidence  $\rank(X^R)=\rank(X)$. By our assumption $\rank(X)$ is finite. Hence, also $\operatorname{rank}(X^R)$ is finite. Then Theorem \ref{t:FiniteRankTame} guarantees that the dynamical $G$-system $X^R$ is Rosenthal representable. 	 
 \end{proof}
 
\subsection*{Roller-Fioravanti compactifications for topological median $G$-spaces}

Let $(X,m,\t)$ be a topological median algebra. 
For every $x,y \in X$ consider the continuous median retraction $\phi_{x,y} \colon X \to [x,y], \phi(z)=  m(x,y,z)$.  
If all intervals $[x,y]$ in $X$ are $\t$-compact then the following diagonal map 
$$
\nu \colon X \to \prod \{[x,y]: x,y \in X\}, \ \ z \mapsto (m(x,y,z))_{x,y} 
$$ 
leads to the median preserving compactification map  $\nu \colon X \to X^{RF}=cl(\nu(X))$.  
This map is injective, continuous and sometimes is said to be a (generalized) \textit{Roller compactification} of $X$; see \cite{Fio19, Fioravanti20}. Perhaps one may call it \textit{Roller-Fioravanti compactification} (RF, in short). 

The compactification $\nu \colon X \to X^{RF}$, in general, is not necessarily a topological embedding. This happens for example for the metric fan (metrizable separable hedgehog $J(\aleph_0)$ with countably many thorns) which is a rank-1 $\R$-tree. In this case the topology inherited from its RF-compactification is compact and $\nu$ is onto.  

Let $(X, d)$ be a complete median metric space of finite rank which is connected and locally
compact. Then $\nu \colon X \to X^{RF}$ is a topological embedding, \cite{Fioravanti20}. 
  
 Recent results from \cite{Me-MedIntr} show that the topology on $\nu(X)$ is an algebraically defined \textit{intrinsic topology} generated by canonical  retractions on chain-intervals.   

The RF-compactification $X^{RF}$ is metrizable in a large class of examples. Indeed, in
Fioravanti's framework, if $X$ is a complete locally convex median metric space with compact
intervals, then $X^{RF}$ is metrizable whenever $X$ is separable,
\cite[Theorem~4.14(4)]{Fioravanti20}.

\subsection*{Joint continuity of the extended action on RF-compactifications}

Denote by $\Ucal_w$ the induced precompact uniformity on $X$. This is the weak (initial) uniformity induced by the family of all canonical interval retractions. 
	
	If $X$ is a median $G$-algebra then $g(xyz)=g(x)g(y)g(z)$ and $[gx,gy]=g[x,y]$ for every $g,x,y \in G \times X \times X$. Then the $g$-translations $X \to X$ are uniformly continuous with respect to the uniformity $\Ucal_w$. This guarantees that there exists a natural action $\pi_* \colon G \times X^{RF} \to X^{RF}$ with continuous $g$-translations such that $\nu$ is a $G$-map. 
	Moreover, this action preserves the median of $X^{RF}$. 
	This means that $X^{RF}$ is a median  $G_d$-algebra and hence the RF-compactification always is at least a (injective, continuous) $G_{d}$-compactification of $X$.  

Our aim is to show that the extended action $\pi_*$  remains jointly continuous. 

From the standard theory of $G$-compactifications and $G$-completions we know that it is enough (in fact, equivalent) to show that $\Ucal_w$ is an \textit{equiuniformity}. 
This reduction to equiuniformities is well known in the theory of $G$-compactifications. The uniform completion of an equiuniform action is a well-defined equiuniform jointly continuous action. 
See, for example, \cite[Section 4]{Me-FL} or \cite{Me-b}. 

In our settings, when $g$-translations are uniformly continuous, equiuniformity of $\Ucal_w$ with respect to the action of $G$ means that $\Ucal_w$ is  \textit{$G$-bounded}. More precisely, recall that 
the latter means that the following condition holds: 

\noindent (\textbf{G}-\textbf{Boundedness}) \ \ For every uniform cover $P \in \Ucal_w$ there exists a neighborhood $V$ of the identity $e \in G$ such that the cover $\{Vx: x \in X\}$ refines $P$.  
\sk  
It is enough to check this condition for every $P \in \Gamma$, where $\Gamma$ is a uniform subbase of $\Ucal_w$. 

\begin{prop} \label{t:G-bound} 
	Let $X$ be a topological median algebra with compact intervals, equipped with a continuous action of a topological group $G$ by median automorphisms. Then the canonical extended action of $G$ on the RF-compactification $X^{RF}$ is jointly continuous.
\end{prop}  
\begin{proof}
The RF-compactification $\nu \colon X \to X^{RF}$ is the uniform completion of $X$ with respect to the weak precompact uniformity $\mathcal{U}_w$ induced by the family of all canonical interval retractions  $\phi_{u,v} \colon X \to [u,v]$ for $u, v \in X$.  
	For any distinct elements $c, d \in X$, consider the branches $B_d^c = \{x \in X \mid m(c,x,d) \neq d\}$ and $B_c^d = \{x \in X \mid m(c,x,d) \neq c\}$. Because $c \neq d$, there is no point $x \in X$ such that $m(c,x,d)$ equals both $c$ and $d$ simultaneously. Thus, $B_d^c \cup B_c^d = X$, and $\mathcal{C}_{c,d} = \{B_d^c, B_c^d\}$ forms a two-element open cover of $X$.  
	
	For any non-singleton compact interval $[u,v] \subseteq X$, and any distinct points $c, d \in [u,v]$, the interval $[c,d]$ is a compact subset of $[u,v]$. The cover $\mathcal{C}_{c,d}$ restricts to a relatively open cover $\{B_d^c \cap [u,v], B_c^d \cap [u,v]\}$ of $[u,v]$ which separates the points $c$ and $d$. Because $[u,v]$ is a compact Hausdorff space, any point-separating family of open covers generates its unique compatible uniformity. 
	 
	For $c,d\in [u,v]$, the pullback of this cover under $\phi_{u,v}$ is exactly $\mathcal B_{c,d}$. Indeed, since $\phi_{u,v}$ is MP and fixes
	$[u,v]$, for every $x\in X$ we have
	$m(c,\phi_{u,v}(x),d)=\phi_{u,v}(m(c,x,d))=m(c,x,d)$,
	because $m(c,x,d)\in[c,d]\subset[u,v]$. Hence
	\[
	\phi_{u,v}^{-1}(B_c^d\cap[u,v])=B_c^d,\qquad
	\phi_{u,v}^{-1}(B_d^c\cap[u,v])=B_d^c.
	\] 
	Consequently, the initial uniformity $\mathcal{U}_w$ on $X$, being generated by the pullbacks from all such compact intervals, admits the family of two-element covers 
	$$ \mathcal{B} = \big\{\mathcal{B}_{u,v} \mid u, v \in X, u \neq v \big\}, \ \ \  \mathcal{B}_{u,v}:=\{B_v^u, B_u^v\}$$
	as a uniform prebase, where $B_v^u = \{x \in X \mid m(u,x,v) \neq v\}$ and $B_u^v = \{x \in X \mid m(u,x,v) \neq u\}$.

\textit{(G-Boundedness)} 
	To prove that $\mathcal{U}_w$ is $G$-bounded, it suffices to prove that each such subbasic cover is $G$-bounded. 
	Fix a subbasic cover $\mathcal{B}_{u,v}$ with $u \neq v$. Assume, for a contradiction, that $B_{u,v}$ is not $G$-bounded. This means that for every neighborhood $V$ of the identity $e \in G$, there exists an element $x \in X$ such that the orbit $V x$ is neither contained in $B_v^u$ nor in $B_u^v$. 
	
	Consequently, there exists a directed set $I$, nets $(g_i)_{i \in I}$ and $(h_i)_{i \in I}$ in $G$ both converging to $e$, and a net $(x_i)_{i \in I}$ in $X$ such that for all $i \in I$:
	$$
		g_i x_i \notin B_v^u \quad \text{and} \quad h_i x_i \notin B_u^v.
	$$
	By the definition of the branches, their complements yield strict equalities:
	$$
		m(u, g_i x_i, v) = v \quad \text{and} \quad m(u, h_i x_i, v) = u.
$$
	Because $G$ acts by median automorphisms, we can apply $g_i^{-1}$ to the first equation and $h_i^{-1}$ to the second equation, yielding:
	\begin{align} 
		m(g_i^{-1}u, x_i, g_i^{-1}v) &= g_i^{-1}v \label{eq:g_inverse} \\
		m(h_i^{-1}u, x_i, h_i^{-1}v) &= h_i^{-1}u \label{eq:h_inverse}.
	\end{align}
	
	Let $\phi_{u,v} \colon X \to [u,v]$ be the canonical gate retraction defined by $\phi_{u,v}(x) = m(u,x,v)$.Then $\phi_{u,v}$ is a median homomorphism. Applying $\phi_{u,v}$ to both sides of equation (\ref{eq:g_inverse}), we obtain:
	$$
		m\big(\phi_{u,v}(g_i^{-1}u), \phi_{u,v}(x_i), \phi_{u,v}(g_i^{-1}v)\big) = \phi_{u,v}(g_i^{-1}v).
	$$
	To simplify notation, let $u_{g,i} = \phi_{u,v}(g_i^{-1}u)$, $v_{g,i} = \phi_{u,v}(g_i^{-1}v)$, and $z_i = \phi_{u,v}(x_i)$. Because $z_i \in [u,v]$ for all $i$, and $[u,v]$ is compact in $X$, there exists a subnet $(z_j)_{j \in J}$ converging to some $z^* \in [u,v]$.
	
	Since the group action is continuous, $g_j^{-1}u \to u$ and $g_j^{-1}v \to v$. Furthermore, because $X$ is a topological median algebra, $\phi_{u,v}$ is continuous. Since $\phi_{u,v}(u) = u$ and $\phi_{u,v}(v) = v$, we have:
	$$
		u_{g,j} \to u \quad \text{and} \quad v_{g,j} \to v.
	$$
	Taking the limit along the subnet $J$ in the equation $m(u_{g,j}, z_j, v_{g,j}) = v_{g,j}$, the joint continuity of the median operation implies:
	\begin{equation} \label{eq:limit_v}
		m(u, z^*, v) = v. 
	\end{equation}
	
	By a symmetric argument, applying $\phi_{u,v}$ to equation (\ref{eq:h_inverse}), defining $u_{h,i} = \phi_{u,v}(h_i^{-1}u)$ and $v_{h,i} = \phi_{u,v}(h_i^{-1}v)$, and passing to the same subnet $J$ where $z_j \to z^*$, we obtain:
	\begin{equation}
		m(u, z^*, v) = u. \label{eq:limit_u}
	\end{equation}
	
 Combining equations (\ref{eq:limit_v}) and (\ref{eq:limit_u}) yields $v = u$. This contradicts the assumption that $u \neq v$. 
	
Thus, the cover $\mathcal{B}_{u,v}$ is $G$-bounded. Since $\Ucal_w$ is generated by such $G$-bounded covers, $\Ucal_w$ itself is $G$-bounded. 
Because the RF-compactification coincides with the uniform completion of $(X, \Ucal_w)$, it follows that the extended action of $G$ on $X^{RF}$ is jointly continuous.
\end{proof}
 
Note that in \cite{Me-MedIntr} we prove a result which is similar to Proposition \ref{t:G-bound} but for median actions on median algebras equipped with the \textit{intrinsic uniform structure}. These two results, being similar, are however incomparable.  

\begin{cor} \label{t:RF-tame} 
	Let $X$ be a topological median $G$-algebra with finite rank and compact intervals, where $G$ is a topological group. Then the canonically defined RF-compactification $\nu \colon X \to X^{RF}$ is a $G$-compactification which is Rosenthal representable (and dynamically $G$-tame). 
\end{cor}
\begin{proof}
As we already know,   
	$\nu \colon X \to X^{RF}$ is an  injective continuous compactification.  Moreover the following two conditions are satisfied: 
	\begin{enumerate}
		\item $X^{RF}$ is a compact median algebra.
		\item The rank of $X^{RF}$ is finite (and equal to the rank of $X$ by Fact \ref{f:facts2}.5). 
	\end{enumerate}
	
	The joint continuity of the $G$-action on $X^{RF}$  follows from Proposition \ref{t:G-bound}. So, $X^{RF}$ is a well defined compact $G$-space.   
	Since $\operatorname{rank}(X^{RF})$ is finite, Theorem \ref{t:FiniteRankTame} implies that the system $(G,X^{RF})$ is Rosenthal representable. It is also dynamically tame by Fact \ref{f:RosImpliesTame}. 
\end{proof}

\begin{cor} \label{t:RF} 
	Let $(X, d)$ be a complete median metric space of finite rank.  
	Let a topological group $G$ act on $X$ continuously by isometries. Then the Roller--Fioravanti compactification $X^{RF}$ is a Rosenthal representable $G$-system 
	(with continuous action) and dynamically $G$-tame. 
\end{cor}
\begin{proof} Every complete median metric space of finite rank has compact intervals, \cite[Corollary 2.20]{Fioravanti20}. Therefore, the RF-compactification $\nu \colon X \to X^{RF}$ is well defined.  
	Since the action of $G$ on $(X,d)$ is isometric it is also median preserving. Hence, $X$ is a median $G$-space.  
	The rest of the proof directly follows now from Corollary \ref{t:RF-tame}.  
\end{proof}

The finite-rank assumption together with compactness of intervals make several natural classes of median \(G\)-spaces directly amenable to Theorem~\ref{t:FiniteRankTame}, Corollary~\ref{t:G-Rol}, Proposition~\ref{t:G-bound}, and Corollaries~\ref{t:RF-tame}--\ref{t:RF}. So, many RF-compactifications (hence, also the corresponding RF-\textit{boundaries} $X^{RF} \setminus X$) are tame $G$-spaces.   

By \cite[Proposition 4.21]{Fioravanti20}, under the hypotheses of Corollary~\ref{t:RF} the RF-compactification coincides with the horofunction (Busemann) compactification of $(X,d)$. Since in this case the horofunction compactification 
is tame, Corollary \ref{t:RF} gives a partial answer to a question posed in \cite[Question 6.7]{Me-FL}.
This applies in particular in the case of finite dimensional CAT(0) cube complexes, which is a major source of median spaces in geometric group theory.

\subsection*{More remarks} 

\begin{remark}\label{r:Generalization}
	Corollary~\ref{t:G-Rol} remains true for every subinfinite-rank median algebra $X$.
	Indeed, the proof uses only that the family of halfspaces of $X$ is Boolean-tame, which
	is equivalent to subinfinite rank by Remark~\ref{r:B-tame}. Let
	$
	\iota \colon X \to \overline X \subseteq \{0,1\}^{H(X)}
	$
	be the Roller compactification, and for each $H\in H(X)$ let
	$
	\widehat\chi_H \colon \overline X \to \{0,1\}
	$ 
	be the corresponding coordinate map. Then $\widehat\chi_H$ is continuous and
	$\widehat\chi_H|_X=\chi_H$. If $(\widehat\chi_{H_n})$ were an independent sequence on
	$\overline X$, then every finite Boolean combination of the clopen sets
	$\widehat\chi_{H_n}^{-1}(1)$ would be a nonempty open subset of $\overline X$, hence
	would meet the dense subset $\iota(X)$. It follows that $(\chi_{H_n})$ would be an
	independent sequence on $X$, contradicting Boolean-tameness. Thus the family
	$\{\widehat\chi_H : H\in H(X)\}$ is tame on $\overline X$. Since it is $G$-invariant and
	separates points of $\overline X$, Fact \ref{WRNcriterion}  implies that the compact
	$G_d$-system $\overline X$ is Rosenthal representable. 
\end{remark}

To appreciate the role of finite rank (or finite Rosenthal dimension), consider the Cantor cube $K = \{0,1\}^\N$ with the product topology and the coordinate-wise median structure.
The group $G = \Aut(K)$ is very large; it contains the group of all permutations of coordinates $S_\infty$ and, if indexed by $\Z$, the shift automorphism (which generates the Bernoulli shift). 

\begin{prop}
	The Cantor cube $K = \{0,1\}^\N$ is a compact median $G$-algebra which is \textbf{not} dynamically tame. Also it is not a subinfinite-rank median algebra.
\end{prop}

\begin{proof}
	Consider the coordinate projections $\pi_n \colon K \to \{0,1\} \subset \R$, defined by $\pi_n(x) = x_n$.
	These maps are continuous and median-preserving. The sequence $\{\pi_n\}_{n=1}^\infty$ is an \textbf{independent sequence}. 
	To see this, let $P, M$ be any two disjoint finite subsets of $\N$. We must find a point $x \in K$ such that:
	\[ \pi_n(x) = 0 \text{ for } n \in P \quad \text{and} \quad \pi_m(x) = 1 \text{ for } m \in M. \]
	Since the coordinates in a product space can be chosen arbitrarily, such a point $x$ clearly exists (set $x_k = 0$ if $k \in P$, $x_k = 1$ if $k \in M$, and arbitrarily otherwise).   
	
	Now, consider the orbit of the first projection $\pi_1$ under the action of $G$. Since $G$ acts transitively on the coordinates (via permutations), the orbit of $\pi_1$ contains the entire set $\{\pi_n\}_{n=1}^\infty$. Since this orbit contains an independent sequence, the function $\pi_1$ is not tame. Consequently, the system $(G, K)$ is not tame. It is true even for the subgroup of $G$ indexed by $\Z$ (where $K$ becomes the Bernoulli shift). 
\end{proof}

\begin{remark} 
Every compact (Hausdorff) space $K$ embeds into a compact median algebra $Y$ which is locally convex. Indeed, take for example, $Y:=[0,1]^{\kappa}$. 
 However, it is not true in general if we require finiteness of the rank for $Y$. 
 Indeed, such embedding is impossible for infinite-dimensional compact space $K$ because   \cite[Lemma 12.3.3]{BowditchMedian} asserts that for a compact median algebra $Y$ we have $\dim(Y) \leq \rank(Y)$. 
 
 For nonmetrizable $K$ even zero-dimensionality is not a sufficient condition.  
 Indeed, let $K$ be a compact space which is not WRN. Then by Theorem \ref{t:FiniteRankTame} $K$ cannot be embedded into a finite rank compact median algebra. This happens, for example, for the $0$-dimensional space $K:=\beta \N$ the Stone-\v{C}ech compactification of $\N$; see an argument of Todorcevic presented in \cite{GM-TC} which shows that $\beta \N$ is not WRN.  
\end{remark}

It would be interesting to understand which additional dynamical or structural restrictions arise in the presence of finite rank.

\begin{question} \label{q:GrRos} 
	Which topological groups $G$ can be embedded into the automorphism group $\Aut(X)$ (compact-open topology) for some finite-rank compact median space $X$?
\end{question} 

An additional motivation for Question \ref{q:GrRos} is Theorem \ref{t:FiniteRankTame}. Recall that it remains an open question whether every topological (say, Polish) group is Rosenthal representable.

The following definition from \cite{GM-TC} is justified by Todor\u{c}evi\'{c}'s Trichotomy and the dynamical version of the Bourgain-Fremlin-Talagrand dichotomy.  

\begin{defin} \label{d:TameClasses} 
	A compact metrizable dynamical $G$-system is said to be:
	\begin{enumerate}
		\item Tame$_\mathbf{1}$ if $E(G,X)$ is first countable.
		\item Tame$_\mathbf{2}$ if $E(G,X)$ is hereditarily separable. 
	\end{enumerate}	
	By results of \cite{GM-TC} we know that 
	$
	\mathrm{Tame}_\mathbf{2} \subset \mathrm{Tame}_\mathbf{1} \subset \mathrm{Tame}. 
	$ 
\end{defin}

Since every compact finite rank median $G$-space is tame, in view of the hierarchy of tame dynamical systems established in \cite{GM-TC}, we propose the following general problem. 

\begin{question} \label{q:FineClassification} 
	Which natural finite rank compact metrizable median $G$-algebras 
	$K$ are Tame$_1$ ?  Tame$_2$ ? 
\end{question}

One may show that finite-dimensional cubes $K:=[0,1]^n$ as compact median $G$-spaces are Tame$_\textbf{1}$ with respect to the action of the Polish group $G:=\Aut(K)$ of all homeomorphic median automorphisms. 
In contrast, note that by a result of Codenotti \cite{Codenotti}, for the  \textit{Wazewski dendrite} $W$ (which is an example of rank 1 compact median algebra) the corresponding $G$-system $W$ with $G=\Aut(W)$ is not Tame$_1$ (although it is tame by Theorem \ref{t:FiniteRankTame}).  

\begin{ex} \label{r:CUBES} Here we give some examples: 
\begin{enumerate}
	\item \cite{GM-TC} 
Consider the linearly ordered $H_+([0,1])$-system $[0,1]$, where $H_+[0,1]$ is the group of all order preserving homeomorphisms of $[0,1]$. The enveloping semigroup $E(G,X)$ of this order preserving system is a (compact) subspace of the Helly space, which is first countable. So, this system is Tame$_\mathbf{1}$. It is not Tame$_{\bf 2}$. 
	In fact, $E(G,X)$ (like the Helly space) is not hereditarily separable. 
	
	\item Consider $K:=[0,1]$ as a linear compact median algebra and $G=\Aut(K)=\Homeo[0,1]$ the topological group of all monotone homeomorphisms. Then this dynamical $G$-system is Tame$_1$.  
	Indeed, the enveloping semigroup is contained in the union of two compact spaces: the space of all nondecreasing and the space of all nonincreasing self-maps 
	$[0,1]\to[0,1]$, endowed with the pointwise
	topology. Both are first countable Helly-type compacta.


\item 
Let $K=[0,1]^n$ and $G=\Aut(K)$. Put $H:=\Aut([0,1])$. By \cite[Example 3.8]{BowditchMedian}, every
$g\in G$ can be written uniquely in the form
\[
g(x_1,\dots,x_n)=\bigl(h_1(x_{\sigma(1)}),\dots,h_n(x_{\sigma(n)})\bigr)
\]
for some permutation $\sigma\in S_n$ and some $h_1,\dots,h_n\in H$. Hence$G$ is the semidirect product
$
G\cong H^n\rtimes S_n.
$  
For every $\sigma\in S_n$, define
\[
\Phi_\sigma \colon (E(H,[0,1]))^n\to K^K,
\ \ 
\Phi_\sigma(f_1,\dots,f_n)(x_1,\dots,x_n)
:=
\bigl(f_1(x_{\sigma(1)}),\dots,f_n(x_{\sigma(n)})\bigr).
\]
Then
$
E(G,K)=\bigcup_{\sigma\in S_n}\Phi_\sigma\bigl((E(H,[0,1]))^n\bigr).
$
By (2), $E(H,[0,1])$ is first countable. Therefore $(E(H,[0,1]))^n$ is first countable, and each
set $\Phi_\sigma\bigl((E(H,[0,1]))^n\bigr)$ is a compact first countable subspace of $K^K$.
Since $S_n$ is finite, $E(G,K)$ is a finite union of closed first countable subspaces, hence is
itself first countable. Thus the dynamical system $(G,K)$ is Tame$_\textbf{1}$.   
\end{enumerate}	
\end{ex}


\section{Finite-rank median cascades satisfy Sarnak's M\"obius Disjointness conjecture} 
\label{s:Sarnak} 

Let \(K\) be a compact metrizable space and let \(T\colon K\to K\) be a
continuous selfmap. We say that the system \((K,T)\) satisfies
\textit{Sarnak's M\"obius Disjointness Conjecture} (SMDC) if for every
\(\varphi\in C(K)\) and every \(x\in K\),
\[
\lim_{N\to\infty}
{1\over N}\sum_{n=1}^{N}\mu(n)\varphi(T^n x)=0,
\]
where \(\mu\colon \mathbb N\to \{-1,0,1\}\) is the M\"obius function:
\(\mu(1)=1\), \(\mu(n)=(-1)^t\) if \(n\) is a product of \(t\) distinct
primes, and \(\mu(n)=0\) otherwise.

Sarnak's conjecture predicts this property for every compact metrizable
system of zero topological entropy.  
This problem was extensively studied in many  publications. See, for example, \cite{Sarnak,HWY,AD,LOZ,AL,AN} and also a survey \cite{FKL} (and references therein).  
Here we present only results which are strongly related to the setting of the present work. Especially, concerning tame $\Z$-cascades. 

Recall also that every tame compact $\Z$-cascade has zero topological
entropy, by Kerr--Li \cite{KL}. Thus tame cascades form a natural
subclass of the zero-entropy systems appearing in Sarnak's conjecture.
In this context, the following result of Huang--Wang--Ye is particularly
relevant.



\begin{f} \label{f:HWY} 
	(Huang--Wang--Ye \cite[Corollary 1.5]{HWY})
	Every tame compact metrizable invertible cascade satisfies SMDC.
\end{f}

It is well known that weakly almost periodic systems are tame. 
More generally, every hereditarily nonsensitive (HNS) system is tame; in particular, every RN (Asplund representable) system is tame, \cite{GM-survey14}. 
 This already gives many suitable examples. 
 
 We note some tame but not HNS examples. By a result of Karagulyan
 \cite{Kar}, every orientation-preserving homeomorphism of the circle satisfies
 SMDC. This can be generalized in several directions.
 Indeed, for every circularly ordered compact metric space \(X\) and every
 circular order-preserving homeomorphism \(T \colon X\to X\), the cascade
 \((X,T)\) is tame \cite{GM-c,GM-tLN}; hence, by Fact~\ref{f:HWY}, it satisfies SMDC.
 This includes, for example, Sturmian cascades, such as the subshift $X \subset \{0,1\}^{\Z}$ generated
 by the Fibonacci substitution.
 
 Moreover, by \cite[Theorem 2.3]{GM-D}, every continuous action of a
 group \(G\) on a regular continuum is tame. Here, \textit{regularity} is
 understood in the continuum-theoretic sense: a continuum is regular
 (equivalently, rim-finite) if each point has a local base consisting of
 open sets with finite boundaries. Every local dendrite is regular.
 
The Belk--Forrest construction provides another source of examples:
every connected limit space associated with an expanding edge-replacement
system is a regular continuum. Indeed, the cell interiors and open stars
constructed in \cite[Lemma 1.26 and the proof of Theorem 1.25]{BF19}
yield local bases with finite boundaries; see also
\cite[Proposition 1.29]{BF19}. This class contains many familiar
self-similar fractals, including the Vicsek fractal and several Julia-set
limit spaces \cite{BF19}. 
As pointed out to the author by M. Tarocchi, the basilica, the Douady rabbit
and the airplane Julia sets form a closely related but conceptually distinct
family of examples. Although these Julia sets can also be realized as limit
spaces of edge-replacement systems, their rim-finiteness can also be obtained
from the intrinsic analysis of their laminations using 
\cite[Corollaries 1.10 and 2.3]{DucTar26}.

The following theorem emphasizes the relevance of tame cascades as a useful sufficient condition. 

\begin{thm} \label{t:SMDC} 
	Let $K$ be a compact metrizable median algebra of finite rank and let
	$T \colon K\to K$ be a homeomorphic median automorphism of $K$. Then the cascade
	$(K,T)$ satisfies SMDC. 
\end{thm}
\begin{proof}
	The homeomorphism $T$ defines a continuous action of the discrete group
	$\Z$ on $K$ by median automorphisms. By Theorem~\ref{t:FiniteRankTame}, the
	cascade $(K,T)$ is dynamically tame. Hence, by Fact~\ref{f:HWY}, it
	satisfies SMDC. 
\end{proof}

For a dendrite $K$ every homeomorphism $T \colon K \to K$ automatically is median preserving.  
This implies the particular case where $K$ is a dendrite, \cite{GM-D}. 

\begin{ex}[Minimal median compactifications]
	\label{ex:MMC-Mobius}
	Using the intrinsic median compactification from \cite{Me-MedIntr}, one obtains
	the following additional source of M\"obius-disjoint systems. 
\begin{enumerate} 
\item 	Let \(X\) be a finite-rank median algebra such that its intrinsic median topology \(\tau_m\) is separable and metrizable,
and let \(\widehat X\) be its Minimal Median Compactification (MMC) in the sense of \cite{Me-MedIntr}. Then
\(\widehat X\) is a compact metrizable finite-rank median algebra.  
Every median automorphism \(T\in\Aut(X,m)\)
extends to a homeomorphic median automorphism \(\widehat T\) of \(\widehat X\). Hence the cascade
\((\widehat X,\widehat T)\) satisfies SMDC by Theorem  \ref{t:SMDC}.

\item Let \(X\) be a countable finite-rank median algebra with finite intervals.
By \cite{Me-MedIntr}, the MMC of \(X\) coincides with the Roller
compactification \(X^R\), and \(X^R\) is metrizable. 
Therefore every median automorphism \(T\in\Aut(X)\) induces a
M\"obius-disjoint cascade \((X^R,T^R)\), by Theorem~\ref{t:SMDC}.
	\end{enumerate}  
\end{ex}

\begin{ex}[Cubulated groups]
	\label{ex:cubulated-groups-Mobius}
	Let \(X\) be the vertex median algebra of a countable finite dimensional
	locally finite CAT(0) cube complex. Then \(X\) has finite rank and finite
	intervals. Hence, by \cite{Me-MedIntr}, its Roller compactification \(X^R\)
	is a compact metrizable finite-rank median algebra. If a group \(G\) acts on
	\(X\) by cubical automorphisms, then Theorem \ref{t:SMDC} guarantees again that for every \(g\in G\) the induced cascade
	\((X^R,g)\) satisfies SMDC.
	
	This applies, in particular, to right-angled Artin groups associated with
	finite graphs: such a group acts freely and cocompactly by cubical 
	automorphisms on the universal cover of its Salvetti complex; see, for example,
	\cite{CharneyRAAG}. It also applies to right-angled Coxeter groups acting on
	finite-dimensional Davis cube complexes, to fundamental groups of compact
	nonpositively curved cube complexes, to finite products of countable locally finite
	trees, and more generally to finite-dimensional Sageev cubulations whose
	vertex median algebras are countable and locally finite. 
	In the RAAG case, Corollary~\ref{t:RF} also gives tameness of the induced
	action on the RF-compactification; whenever this compactification is
	metrizable, individual elements give SMDC cascades by Theorem~\ref{t:SMDC}.
\end{ex}

\begin{remark} \label{r:SemigroupActions3}
	By Remark~\ref{r:SemigroupActions2}, finite-rank tameness extends to
	semigroup actions by median endomorphisms. The corresponding SMDC conclusion
	is not obtained here, since the Huang--Wang--Ye criterion used in
	Fact~\ref{f:HWY} is proved for invertible cascades. It is natural to ask
	whether it remains valid for compact metrizable \(\mathbb N_0\)-systems; a
	positive answer would extend Theorem~\ref{t:SMDC} to continuous median
	endomorphisms.
\end{remark}

%
%
%
%
%
%
%
%

   \bibliographystyle{amsplain}

\end{document}